\documentclass[12pt]{article}
\usepackage{color}

\usepackage{amsmath}
\usepackage{amsfonts}
\usepackage{amsthm}
\usepackage{amssymb}
\usepackage{bm}
\usepackage[applemac]{inputenc}
\usepackage{enumerate}
\makeatletter

\newtheoremstyle{mystyle}{}{}{\slshape}{2pt}{\scshape}{.}{ }{} 

\newtheorem{thm}{Theorem}[section]
\newtheorem{cor}[thm]{Corollary}

\newtheorem{prop}[thm]{Proposition}
\newtheorem{lemma}[thm]{Lemma}
\newtheorem{fact}[thm]{Fact}

\theoremstyle{definition}
\newtheorem{defi}[thm]{Definition}
\theoremstyle{mystyle}

\theoremstyle{remark}
\newtheorem{rem}[thm]{Remark}

\newcommand{\monster}{\mathfrak C}

\DeclareMathOperator{\tp}{tp}

\title{A note on generically stable measures and $fsg$ groups}
\author{Ehud Hrushovski\thanks{Supported  by ISF}\\Hebrew University of Jerusalem \and Anand Pillay \thanks{Supported by EPSRC grants EP/F009712/1 and EP/I002294/1}\\ University of Leeds \and Pierre Simon \\ENS and Univ. Paris-Sud}

\begin{document}
\maketitle

\begin{abstract} We prove (Proposition 2.1) 
that if $\mu$ is a generically stable measure  in an $NIP$ theory, and $\mu(\phi(x,b)) = 0$ for all $b$ then for some $n$, 
$\mu^{(n)}(\exists y(\phi(x_{1},y)\wedge .. \wedge \phi(x_{n},y))) = 0$. As a consequence we show (Proposition 3.2) that if $G$ is a definable group with $fsg$ in an $NIP$ theory, and $X$ is a definable subset of $G$ then $X$ is generic if and only if every translate of $X$ does not fork over $\emptyset$, precisely as in stable groups, answering positively Problem 5.5 from \cite{NIP2}.

\end{abstract}

\section{Introduction and preliminaries}
This short paper is a contribution to the generalization of stability theory and stable group theory to $NIP$ theories, and also provides another example where we need to resort to measures to prove statements (about definable sets and/or types) which do not explicitly mention measures. The observations in the current paper can and will be used in the future to sharpen existing results around measure and $NIP$ theories (and this is why we wanted to record the observations here). Included in these sharpenings will be:  (i) replacing average types by generically stable types in a characterization of strong dependence in terms of measure and weight in \cite{strongdep}, and (ii) showing the existence of ``external generic types" (in the sense of Newelski \cite{Newelski}), over any model, for $fsg$ groups in $NIP$ theories, improving on Lemma 4.14 and related results from \cite{Newelski}. 

If $p(x)\in S(A)$ is a stationary type in a stable theory and $\phi(x,b)$ any formula, then we know that $\phi(x,b)\in p|\monster$ if and only if $\models \bigwedge_{i=1,,.n}\phi(a_{i},b)$ for some independent realizations $a_{1},..,a_{n}$ of $p$ (for some $n$ depending on $\phi(x,y)$). Hence $\phi(x,b)\notin p|\monster$ for all $b$ implies that
(and is clearly implied by) the inconsistency of $\bigwedge_{i=1,..,n}\phi(a_{i},y)$ for some (any) independent set 
$a_{1},..,a_{n}$ of realizations of $p$. This also holds for generically stable types in $NIP$ theories (as well as for generically stable types in arbitrary theories, with definition as in \cite{PT}). In \cite{strongdep}, an analogous result was proved for ``average measures" in strongly dependent theories. Here we prove it  (Proposition  2.1) for generically stable measures in arbitrary $NIP$ theories, as well as giving a generalization (Remark 2.2).  

The $fsg$ condition on a definable group $G$ is a kind of ``definable compactness" assumption, and in fact means precisely this in $o$-minimal theories and suitable theories of valued fields (and of course stable groups are $fsg$). Genericity of a definable subset $X$ of $G$ means that finitely many translates of $X$ cover $G$. Proposition 2.1 is used to show that for $X$ a definable subset of an $fsg$ group $G$, $X$ is generic if and only if every translate of $X$ does not fork over $\emptyset$. This is a somewhat striking extension of stable group theory to the $NIP$ environment.

We work with an $NIP$ theory $T$ and inside some monster model $\monster$. If $A$ is any set of parameters, let $L_x(A)$ denote the Boolean algebra of $A$-definable sets in the variable $x$. A \emph{Keisler measure} over $A$ is a finitely additive probability measure on $L_x(A)$. Equivalently, it is a regular Borel probability measure on the compact space $S_x(A)$. We will denote by $\mathfrak M_x(A)$ the space of Keisler measures over $A$ in the variable $x$. We might omit $x$ when it is not needed or when it is included in the notation of the measure itself ({\it e.g.} $\mu_x$).
If $X$ is a sort, or more generally definable set, we may also use notation such $L_{X}(A)$, $S_{X}(A)$, $\mathfrak M_{X}(A)$, where for example $S_{X}(A)$ denote the complete types over $A$ which contain the formula defining $X$  (or which ``concentrate on $X$"). 

\begin{defi}
A type $p\in S_x(A)$ is \emph{weakly random} for $\mu_{x}$ if $\mu(\phi(x))>0$ for any $\phi(x)\in L(A)$ such that $p\vdash \phi(x)$. A point $b$ is weakly random for $\mu$ over $A$ if $\tp(b/A)$ is weakly random for $\mu$.
\end{defi}

We briefly recall some definitions and properties of Keisler measures, referring the reader to \cite{NIP3} for more details.

If $\mu \in \mathfrak M_x(\monster)$ is a global measure and $M$ a small model, we say that $\mu$ is $M$-invariant if $\mu(\phi(x,a) \triangle \phi(x,a'))=0$ for every formula $\phi(x,y)$ and $a,a'\in \monster$ having the same type over $M$. Such a measure admits a Borel defining scheme over $M$: For every formula $\phi(x,y)$, the value $\mu(\phi(x,b))$ depends only on $\tp(b/M)$ and for any Borel $B\subset [0,1]$, the set $\{p\in S_y(M) : \mu(\phi(x,b))\in B \text{ for some }b\models p\}$ is a Borel subset of $S_y(M)$.

Let $\mu_x \in \mathfrak M(\monster)$ be $M$-invariant. If $\lambda_y\in \mathfrak M(\monster)$ is any measure, then we can define the \emph{invariant extension} of $\mu_x$ over $\lambda_y$, denoted $\mu_x \otimes \lambda_y$. It is a measure in the two variables $x,y$ defined in the following way. Let $\phi(x,y) \in L(\monster)$. Take a small model $N$ containing $M$ and the parameters of $\phi$. Define
$\mu_x \otimes \lambda_y (\phi(x,y)) = \int f(p) d\lambda_y,$ the integral ranging over $S_y(N)$
where $f(p) = \mu_x(\phi(x,b))$ for $b\in \monster$, $b\models p$ (this function is Borel by Borel definability). It is easy to check that this does not depend on the choice of $N$.

If $\lambda_y$ is also invariant, we can also form the product $\lambda_y \otimes \mu_x$. In general it will not be the case that $\lambda_y \otimes \mu_x=\mu_x \otimes \lambda_y$.

If $\mu_x$ is a global $M$-invariant measure, we define by induction: $\mu^{(n)}_{x_1...x_n}$ by $\mu^{(1)}_{x_1}=\mu_{x_1}$ and $\mu^{n+1}_{x_1...x_{n+1}} = \mu_{x_{n+1}} \otimes \mu^{(n)}_{x_1...x_n}$. We let $\mu^{(\omega)}_{x_1x_2...}$ be the union and call it the \emph{Morley sequence} of $\mu_x$.
\\

Special cases of $M$-invariant measures include definable and finitely satisfiable measures. A  global measure $\mu_x$ is \emph{definable} over $M$ if it is $M$-invariant and for every formula $\phi(x,y)$ and open interval $I\subset [0,1]$ the set $\{p\in S_y(M) : \mu(\phi(x,b))\in I \text{ for some }b\models p\}$ is open in $S_y(M)$. The measure $\mu$ is \emph{finitely satisfiable} in $M$ if $\mu(\phi(x,b))>0$ implies that $\phi(x,b)$ is satisfied in $M$. Equivalently, any weakly random type for $\mu$ is finitely satisfiable in $M$.

\begin{lemma} Let $\mu \in \mathfrak M_{x}(\monster)$ be definable over $M$, and $p(x)\in S_{x}(\monster)$ be weakly random for $\mu$. Let $\phi(x_{1},..,x_{n})$ be a formula over $\monster$. Suppose that 
$\phi(x_{1},..,x_{n})\in p^{(n)}$. Then $\mu^{(n)}(\phi(x_{1},..,x_{n})) > 0$.
\end{lemma}
\begin{proof} We will carry out the proof in the case where $\mu$ is definable (over $M$), which is anyway the case we need. Note that $p^{(m)}$ is $M$-invariant for all $m$. The proof of the lemma is by induction on $n$. For $n=1$ it is just the definition of weakly random. Assume true for $n$ and we prove for $n+1$. So suppose $\phi(x_{1},..,x_{n},x_{n+1}) \in p^{(n+1)}$. This means that for  $(a_{1},..,a_{n})$ realizing $p^{(n)}|M$, $\phi(a_{1},..,a_{n},x) \in p$. So as $p$ is weakly random for $\mu$, $\mu(\phi(a_{1},..,a_{n},x)) = r>0$. So as $\mu$ is $M$-invariant, $tp(a_{1}',..,a_{n}'/M) = tp(a_{1},..,a_{n}/M)$ implies $\mu(\phi(a_{1}',..,a_{n}',x)) = r$ and thus also
$r-\epsilon < \mu(\phi(a_{1}',..,a_{n}',x))$  for any small positive $\epsilon$. By definability of $\mu$ and compactness there is a formula $\psi(x_{1},..,x_{n}) \in tp(a_{1},..,a_{n}/A)$ such that 
$\models\psi(a_{1}',..,a_{n}')$ implies  $ 0 < r-\epsilon < \mu(\phi(a_{1}',..,a_{n}',x))$. By induction hypothesis, 
$\mu^{(n)}(\psi(x_{1},..,x_{n})) > 0$. So by definition of $\mu^{(n+1)}$ we have that $\mu^{(n+1)}(\phi(x_{1},..,x_{n},x_{n+1})) > 0$  as required. 
\end{proof}

A measure $\mu_{x_1,...,x_n}$ is \emph{symmetric} if for any permutation $\sigma$ of $\{1,...,n\}$ and any formula $\phi(x_1,...,x_n)$, we have $\mu(\phi(x_1,...,x_n))=\mu(\phi(x_{\sigma.1},...,x_{\sigma.n}))$. A special case of a symmetric measure is given by powers of a generically stable measure as we recall now. The following is Theorem 3.2 of \cite{NIP3}:

\begin{fact}\label{genstable}
Let $\mu_x$ be a global $M$-invariant measure. Then the following are equivalent:
\begin{enumerate}
\item $\mu_x$ is both definable and finitely satisfiable (necessarily over $M$),
\item $\mu^{(n)}_{x_1,...,x_n}|_M$ is symmetric for all $n<\omega$,
\item for any global $M$-invariant Keisler measure $\lambda_y$, $\mu_x \otimes \lambda_y=\lambda_y \otimes \mu_x$,
\item $\mu$ commutes with itself: $\mu_x \otimes \mu_y=\mu_y\otimes \mu_x$. 
\end{enumerate}
If $\mu_x$ satisfies one of those properties, we say it is \emph{generically stable}.
\end{fact}

If $\mu\in \mathfrak M_x(A)$ and $D$ is a definable set such that $\mu(D)>0$, we can consider the \emph{localisation} of $\mu$ at $D$ which is a Keisler measure $\mu_D$ over $A$ defined by $\mu_D(X)=\mu(X\cap D)/\mu(X)$ for any definable set $X$.

We will use the notation $Fr(\theta(x),x_1,...,x_n)$ to mean $$\frac 1 n |\{i\in \{1,...,n\} : \models \theta(x_i)\}|.$$

The following is a special case of Lemma 3.4. of \cite{NIP3}.

\begin{prop}\label{genslemma}
Let $\phi(x,y)$ be a formula over $M$ and fix $r\in (0,1)$ and $\epsilon >0$.

Then there is $n$ such that for any symmetric measure $\mu_{x_1,...,x_{2n}}$, we have $$\mu_{x_1,...,x_{2n}}( \exists y (|Fr(\phi(x,y),x_1,....,x_n) - Fr(\phi(x,y),x_{n+1},...,x_{2n}) | > r)) \leq \epsilon.$$
\end{prop}

\section{Main result}

\begin{prop}
Let $\mu_x$ be a global generically stable measure. Let $\phi(x,y)$ be any formula in $L(\monster)$. Suppose that $\mu(\phi(x,b))=0$ for all $b\in \monster$. Then there is $n$ such that $\mu^{(n)}( \exists y (\phi(x_1,y) \wedge ... \wedge \phi(x_n,y)))=0$.

Moreover, $n$ depends only on $\phi(x,y)$ and not on $\mu$.
\end{prop}
\begin{proof}
Let $\mu_x$ be a global generically stable measure and $M$ a small model over which $\phi(x,y)$ is defined and such that $\mu_x$ is $M$-invariant. Assume that $\mu(\phi(x,b))=0$ for all $b\in \monster$. For any $k$, define $$W_k = \{(x_1,...,x_n) : \exists y (\wedge_{i=1..k} \phi(x_i,y))\}.$$ This is a definable set. We want to show that $\mu^{(n)}(W_n)=0$ for $n$ big enough. Assume for a contradiction that this is not the case.

Let $n$ be given by Proposition \ref{genslemma} for $r=1/2$ and $\epsilon=1/2$. Consider the measure $\lambda_{x_1,...,x_{2n}}$ over $M$ defined as being equal to $\mu^{(2n)}$ localised on the set $W_{2n}$ (by our assumption, this is well defined). As the measure $\mu^{(2n)}$ is symmetric and the set $W_{2n}$ is symmetric in the $2n$ variables, the measure $\lambda$ is symmetric. Let $\chi(x_1,...,x_{2n})$ be the formula ``$(x_1,...,x_{2n}) \in W_{2n} \wedge \forall y (|Fr(\phi(x,y),x_1,...,x_n)-Fr(\phi(x,y),x_{n+1},...,x_{2n})| \leq 1/2)$". By definition of $n$, we have
$\lambda ( \exists y (|Fr(\phi(x,y),x_1,....,x_n) - Fr(\phi(x,y),x_{n+1},...,x_{2n}) | > 1/2)) \leq 1/2$. Therefore $\mu^{(2n)} (\chi(x_1,...,x_{2n})) >0$.

As $\mu$ is $M$-invariant, we can write

$$\mu^{(2n)}(\chi(x_1,...,x_{2n})) = \int_{q \in S_{x_1,...,x_n}(M)} \mu^{(n)} (\chi(q,x_{n+1},...,x_{2n}))d\mu^{(n)},$$

where $\mu^{(n)} (\chi(q,x_{n+1},...,x_{2n}))$ stands for $\mu^{(n)}(\chi(a_1,...,a_n,x_{n+1},...,x_{2n}))$ for some (any) realization $(a_1,...,a_n)$ of $q$. As $\mu^{(2n)}(\chi(x_1,...,x_{2n})) >0$, there is $q\in S_{x_1,...,x_n}$ such that \\(*) $\mu^{(n)}(\chi(q,x_{n+1},...,x_{2n}))>0$.\\Fix some $(a_1,...,a_n)\models q$. By (*), we have $(a_1,...,a_n) \in W_n$. So let $b\in \monster$ such that $\models \bigwedge_{i=1...n} \phi(a_i,b)$. Again by (*), we can find some $(a_{n+1},...,a_{2n})$ weakly random for $\mu^{(n)}$ over $Ma_1...a_nb$ and such that
\\ (**) $\models \chi(a_1,...,a_n,a_{n+1},...,a_{2n})$.
\\ In particular, for  $j=n+1,...,2n$, $a_j$ is weakly random for $\mu$ over $Mb$ hence $\models \neg \phi(a_j,b)$. But then $|Fr(\phi(x,b);a_1,...,a_n) - Fr(\phi(x,b);a_{n+1},...,a_{2n})| =1$. This contradicts (**).
\end{proof}

\begin{rem} The proof above adapts to showing the following generalization:
\newline
Let $\mu_{x}$ be a global generically stable measure, $\phi(x,y)$ a formula in $L(\monster)$. Let $\Sigma(x)$ be the partial type  (over the parameters in $\phi$ together with a small model over which $\mu$ is definable) defining $\{b:\mu(\phi(x,b)) = 0\}$. Then for some $n$: $\mu^{(n)}(\exists y(\Sigma(y) \wedge \phi(x_{1},y)\wedge .. \wedge\phi(x_{n},y))) = 0$.
\end{rem}

\section{Generics in $fsg$ groups }
Let $G$ be a definable group, without loss defined over $\emptyset$. We call a definable subset $X$ of $G$ left (right) generic if finitely many left (right) translates of $X$ cover $G$, and a type $p(x)\in S_{G}(A)$ is left (right) generic if every formula in $p$ is. We originally defined (\cite{NIP1}) $G$ to have ``finitely satisfiable generics", or to be $fsg$, if there is some global complete type $p(x)\in S_{G}(\monster)$ of $G$ every left $G$-translate of which is finitely satisfiable in some fixed small model $M$. 

The following summarizes the situation, where the reader is referred to Proposition 4.2 of \cite{NIP1} for (i) and Theorem 7.7 of \cite{NIP2} and Theorem 4.3 of \cite{NIP3} for (ii), (iii), and (iv). 

\begin{fact} Suppose $G$ is an $fsg$  group. Then
\newline
(i) A definable subset $X$ of $G$ is left generic iff it is right generic, and the family of nongeneric definable sets is a (proper) ideal of the Boolean algebra of definable subsets of $G$,
\newline
(ii) There is a left $G$-invariant  Keisler measure $\mu\in \mathfrak M_G(\monster)$ which is generically stable,
\newline
(iii) Moreover $\mu$ from (ii) is the unique left $G$-invariant global Keisler measure on $G$ as well as the unique right $G$-invariant global Keisler measure on $G$,
\newline
(iv) Moreover $\mu$ from (ii) is {\em generic} in the sense that for any definable set $X$, $\mu(X) > 0$ iff $X$ is generic.
\end{fact}

Remember that a definable set $X$ (or rather a formula $\phi(x,b)$ defining it) forks over a set $A$ if $\phi(x,b)$ implies a finite disjunction of formulas $\psi(x,c)$ each of which divide over $A$, and  $\psi(x,c)$ is said to divide over $A$ if for some $A$-indiscernible sequence $(c_{i}:i<\omega)$ with $c_{0} = c$, $\{\phi(x,c_{i}):i<\omega\}$ is inconsistent. 

\begin{prop} Suppose $G$ is $fsg$ and $X\subseteq G$ a definable set. Then $X$ is generic if and only if for all $g\in X$, $g\cdot X$ does not fork over $\emptyset$ (if and only if for all $g\in G$, $X\cdot g$ does not fork over $\emptyset$). 
\end{prop}
\begin{proof}  Left to right: It suffices to prove that any generic definable set $X$ does not fork over $\emptyset$, and as the set of nongenerics forms an ideal it is enough to prove that any generic definable set does not divide over $\emptyset$. This is carried out in (the proof of) Proposition 5.12 of \cite{NIP2}. 

\vspace{2mm}
\noindent
Right to left:  Assume that $X$ is nongeneric. We will prove that for some $g\in G$, $g\cdot X$ divides over $\emptyset$ (so also forks over $\emptyset$). 

Let $\mu_{x}$ be the generically stable $G$-invariant global Keisler measure given by Fact 3.1. Let $M_{0}$ be a small model such that $\mu$ does not fork over $M_{0}$ (namely, as $\mu$ is generic, every generic formula does not fork over $M_{0}$) and $X$ is definable over $M_{0}$. Let $\phi(x,y)$ denote the formula defining $\{(x,y)\in G\times G: y\in x\cdot X\}$. So 
$\phi$ has additional (suppressed) parameters from $M_{0}$. Note that for $b\in G$, $\phi(x,b)$ defines the set 
$b\cdot X^{-1}$. As $X$ is nongeneric, so is $X^{-1}$ so also  $b\cdot X^{-1}$ for all $b\in G$. Hence, as $\mu$ is generic, $\mu(\phi(x,b)) = 0$ for all $b$.  By Proposition 2.1, for some $n$ 
$\mu^{(n)}(\exists y(\phi(x_{1},y)\wedge .. \wedge \phi(x_{n},y))) = 0$. Let $p$ be any weakly random type for $\mu$  (which in this case amounts to a global generic type, which note is $M_{0}$-invariant). So by Lemma 1.2 the formula 
$\exists y(\phi(x_{1},y)\wedge .. \wedge \phi(x_{n},y)))\notin p^{(n)}$. Let $(a_{1},..,a_{n})$ realize $p^{(n)}|M_{0}$.
Then $(a_{1},..,a_{n})$ extends to an $M_{0}$-indiscernible sequence $(a_{i}:i=1,2,....)$, a Morley sequence in $p$ over $M_{0}$, and  $\models \neg \exists y(\phi(a_{1},y) \wedge ... \wedge \phi(a_{n},y))$. So  in particular
$\{\phi(a_{i},y):i=1,2,...\}$ is inconsistent. Hence the formula $\phi(a_{i},y)$ divides over $M_{0}$, so also divides over $\emptyset$. But $\phi(a_{1},y)$ defines the set  $a_{1}\cdot X$, so $a_{1}\cdot X$ divides over $\emptyset$ as required.
\end{proof}

Recall that we called a global type $p(x)$ of a $\emptyset$-definable group $G$, left $f$-generic if every left $G$-translate of $p$ does not fork over $\emptyset$.

We conclude the following (answering positively Problem 5.5 from \cite{NIP2} as well as strengthening Lemma 4.14 of \cite{CPI}):

\begin{cor} Suppose $G$ is $fsg$ and $p(x) \in S_{G}(\monster)$. Then the following are equivalent: 
\newline
(i) $p$ is generic,
\newline
(ii) $p$ is left (right) $f$-generic,
\newline
(iii) (Left or right) $Stab(p)$ has bounded index in $G$ (where  left $Stab(p) = \{g\in G:g\cdot p = p\}$).  
\end{cor} 
\begin{proof} The equivalence of (i) and (ii) is given by Proposition 3.2 and the definitions.  We know from \cite{NIP1}, Corollary 4.3, that if $p$ is generic then $Stab(p)$ is precisely $G^{00}$. Now suppose that $p$ is nongeneric. Hence there is a definable set $X\in p$ such that $X$ is nongeneric. Let $M$ be a small model  over which $X$ is defined. Note that the $fsg$ property is invariant under naming parameters. Hence $G$ is $fsg$ in 
$Th(\monster,m)_{m\in M}$. By Proposition 3.2  (as well as what is proved in ``Right to left" there), for some $g\in G$, $g\cdot X$ divides over $M$. As $X$ is defined over $M$ this means that there is an $M$-indiscernible sequence 
$(g_{\alpha}:\alpha < {\bar\kappa})$  (where $\bar\kappa$ is the cardinality of the monster model) and some $n$ such 
that  $g_{\alpha_{1}}\cdot X \cap ... \cap g_{\alpha_{n}}\cdot X = \emptyset$ whenever $\alpha_{1} < ... < \alpha_{n}$.
This clearly implies that among $\{g_{\alpha}\cdot p: \alpha < \bar\kappa\}$, there are $\bar\kappa$ many types, whereby $Stab(p)$ has unbounded index.
\end{proof}


\begin{thebibliography}{99}

\bibitem{CPI} A. Conversano and A. Pillay, Connected components of definable groups and $o$-minimality I, submitted. 

\bibitem{NIP1} E. Hrushovski, Y. Peterzil, and A. Pillay, Groups, measures, and the $NIP$, Journal of AMS, 21 (2008), 563 - 596.

\bibitem{NIP2}  E. Hrushovski and A. Pillay, On $NIP$ and invariant measures, to appear in J. European Math. Soc. 

\bibitem{NIP3} E. Hrushovski, A. Pillay, and P. Simon, Generically stable and smooth measures in NIP theories, to appear in Transactions AMS. 

\bibitem{Newelski}  L. Newelski, Model theoretic aspects of the Ellis semigroup, to appear in Israel J. Mathematics. 

\bibitem{strongdep} A. Pillay, Strong dependence, weight and measures, submitted. 

\bibitem{PT}  A. Pillay and P. Tanovic, Generic stability, regularity and quasiminimality, {\em Models, 
Logics, and Higher-Dimensional Categories: A Tribute to the Work of Mihaly Makkai}, CRM proceedings and Lecture Notes volume 53, AMS, 2011. 

\end{thebibliography}
\end{document}